\documentclass[12pt,reqno]{amsart}
\RequirePackage{amsmath, amsthm, amssymb, amscd, enumitem, verbatim, url}
\title[Wiring switches to light bulbs]{Wiring switches to 
  more light bulbs}
\RequirePackage{graphicx,color,datetime}

\newcommand\ignore[1]{}

                       \def\la{\lambda}

\def\K{\overline{K}}
\def\M{\mathcal{M}}
\def\N{\mathbb{N}}              \def\Z{\mathbb{Z}}
\def\F{\mathbb{F}}
\def\({\left(}           
\def\){\right)}          
\def\lce{\left\lceil}    \def\lfl{\left\lfloor}
\def\rce{\right\rceil}   \def\rfl{\right\rfloor}

\def\deg{\operatorname{deg}}  \def\diag{\operatorname{diag}}
\def\ds{\displaystyle}

\numberwithin{equation}{subsection}
\theoremstyle{plain}
\newtheorem{thm}{Theorem}[section]    \newtheorem{prop}[thm]{Proposition}
\newtheorem{lem}[thm]{Lemma}      \newtheorem{cor}[thm]{Corollary}
 \newtheorem{que}[thm]{Question}
\theoremstyle{remark}

\theoremstyle{plain}
\newtheorem{Thm}{Theorem}
\newtheorem{Lem}[Thm]{Lemma}            
\newtheorem{Cor}[Thm]{Corollary}        %
\author{Stephen M. Buckley} 
\author{Anthony G. O'Farrell}
\address{Department of Mathematics, National University of Ireland Maynooth,
Maynooth, Co.~Kildare, Ireland}
\email{stephen.m.buckley@mu.ie, anthony.ofarrell@mu.ie}
\date{\today:\currenttime}

\thanks{The first author was partly supported by Science Foundation Ireland.
Both authors were partly supported by the European Science Foundation
Networking Programme HCAA}


\catcode`\@=11       

\def\rf#1{\@rf{#1}#1:;;}
\def\rfs#1{\@rfs{#1}#1:;;}
\def\rfm#1{\@rfF#1<>;;}

\def\@C{C}\def\@E{E}\def\@F{F}\def\@L{L}\def\@O{O}\def\@P{P} 
\def\@Q{Q}\def\@R{R}\def\@S{S}\def\@T{T}\def\@X{X}\def\@s{s} 

\def\@rf#1#2:#3;;{\xdef\@b{#2}
  \ifx\@b\@C Corollary~\ref{#1}\else%
  \ifx\@b\@E (\ref{#1})\else
  \ifx\@b\@F Fact~\ref{#1}\else%
  \ifx\@b\@L Lemma~\ref{#1}\else%
  \ifx\@b\@O Observation~\ref{#1}\else%
  \ifx\@b\@P Proposition~\ref{#1}\else%
  \ifx\@b\@Q Question~\ref{#1}\else%
  \ifx\@b\@R Remark~\ref{#1}\else%
  \ifx\@b\@S Section~\ref{#1}\else%
  \ifx\@b\@T Theorem~\ref{#1}\else%
  \ifx\@b\@X Example~\ref{#1}\else%
  \ifx\@b\@s \S\ref{#1}\else
  \ref{#1}\fi\fi\fi\fi\fi\fi\fi\fi\fi\fi\fi\fi}

\def\@rfs#1#2:#3;;{\def\@b{#2}
  \ifx\@b\@C Corollaries~\ref{#1}\else%
  \ifx\@b\@F Facts~\ref{#1}\else%
  \ifx\@b\@L Lemmas~\ref{#1}\else%
  \ifx\@b\@O Observations~\ref{#1}\else%
  \ifx\@b\@P Propositions~\ref{#1}\else%
  \ifx\@b\@Q Questions~\ref{#1}\else%
  \ifx\@b\@R Remarks~\ref{#1}\else%
  \ifx\@b\@S Sections~\ref{#1}\else%
  \ifx\@b\@T Theorems~\ref{#1}\else%
  \ifx\@b\@X Examples~\ref{#1}\else%
  \ifx\@b\@D Definitions~\ref{#1}\else
  \ref{#1}\fi\fi\fi\fi\fi\fi\fi\fi\fi\fi\fi}

\def\@rfF<#1>#2;;{\def\@c{#2}
  \@rfs{#1}#1:;;\ifx\@c\empty\else\@rfL:#2;;\fi}

\def\@rfL:#1<#2>#3;;{\def\@b{#2}\def\@c{#3}
  #1\ifx\@b\empty\else\ref{#2}\ifx\@c\empty\else\@rfL:#3;;\fi\fi}

\catcode`\@=12       

\begin{document}

\begin{abstract}
Given $n$ buttons and $n$ bulbs so that the $i$th button toggles the $i$th
bulb and perhaps some other bulbs, we compute the sharp lower bound on the
number of bulbs that can be lit regardless of the action of the buttons.
In the previous article we dealt with the case where each button affects 
at most 2 or 3 bulbs. In the present article we give sharp lower
bounds for up to 4 or 5 wires per switch, and we show that
the sharp asymptotic bound
for an arbitrary number of wires is $\frac12$.
(Even if you've found their buttons, 
you can please no more than half the people all the time!)
\end{abstract}

\maketitle

\section{Introduction}\label{S:introduction}

\subsection{The function $\mu(m,n)$}
This article is a continuation of \cite{BOF}, and we
refer to that article for motivation and context.
The focus of our attention is the function
$\mu(n,m)$, which counts the minimum number of
bulbs that can always be lit by some switching choice 
when each of $n$ bulbs has a dedicated button ($=$switch) 
that switches it 
and up to $m-1$ other bulbs on or off.  The problem is rephrased
in precise terms using vectors and matrices over $\F_2$,
the field with two elements, as follows:
 
Each conceivable wiring from $n$ buttons to $r$ bulbs
may be represented by an element of
the set $\M(n,r,\F_2)$ of all $n\times r$ matrices over $\F_2$,
by letting column $i$ represent the effect of button $i$.
Replacing $n$ and $r$ by their maximum, and filling in with zeros,
we might as well use square matrices, so for us
a wiring corresponds to a directed graph $G$ on $n$ vertices, represented
by an $n\times n$ matrix $W$ over $\F_2$.  A column vector in $\F_2^n$
may represent either the state (lit or unlit) of the $n$
bulbs, or a choice (press or don't press) for $n$ buttons. 
The effect of switch choice $x$  on state $c$ gives state
$Wx+c$.

We are focussed on wirings with $1$ on the diagonal, and we call these
\emph{admissible wirings}, but we shall have occasional use for 
inadmissible wirings. 

The {\it Hamming norm} $|\cdot|:\F_2^n\to\Z_{\ge 0}$ is 
defined by letting
$|u|$ be the 
the number of $1$ entries in $u$. 
We define
$M(W,c):=\max\{\,|Wx+c|: x\in Z_2^n\,\}$.
This number represents the maximal number of bulbs that
can be lit by a choice of switches, given initial state $c$.

Given a wiring $W$, the \emph{associated degree of vertex $i$}
is the Hamming norm of the $i$-th column of $W$ (the out-degree 
of node $i$ in the graph $G$, the number of bulbs affected by button $i$). 
The degree of $W$ is the maximum
associated degree.

For any $n\in\N$, and any set $A$ of $n\times n$ matrices over $\F_2$,
we define
\begin{align*}
\mu_A &= \min\{ M(W,0) \mid W\in A \}\,, \\
\nu_A &= \min\{ M(W,c) \mid W\in A,\; c\in\F_2^n \}\,\\.
\end{align*}
For $n,m\ge 1$, let $A(n,m)$ be the set of $n\times n$ matrices over $\F_2$ that
have $1$s all along the diagonal and satisfy $\deg(W)\le m$. For $n\ge m\ge
1$, let $A^*(n,m)$ be the set of matrices in $A(n,m)$ for which $\deg(i)=m$,
for all $i\in S$. The class of all admissible wirings on $n$ vertices
is $A(n):=A(n,n)$.

The functions we study are:
\begin{alignat*}{3}
\mu(n,m)&:=\mu_{A(n,m)}\,, \qquad \mu^*(n,m)&:=\mu_{A^*(n,m)}\,, \qquad
\mu(n)&:=\mu(n,n)\,, \\
\nu(n,m)&:=\nu_{A(n,m)}\,, \qquad \nu^*(n,m)&:=\nu_{A^*(n,m)}\,, \qquad
\nu(n)&:=\nu(n,n)\,,
\end{alignat*}
It is convenient to define $\mu(0,m)=0$
for all $m\in\N$.
Given $n\ge m$, we have the following trivial inequalities:
\begin{align}
\nu(n,m) &\le \nu^*(n,m) \le \mu^*(n,m) \\
\nu(n,m) &\le \mu(n,m) \le \mu^*(n,m)
\end{align}

\subsection{Results}
General formulae for $\nu$ and $\nu^*$, and formulae for $\mu(\cdot,m)$ and
$\mu^*(\cdot,m)$ for $m=2,3$ were determined in \cite{BOF}. We'll summarise these 
in Section \ref{S:recap} below, but   
right now we mention only that if $m=2,3$, then $\mu(n,m)$
and $\mu^*(n,m)$ are asymptotic to $2n/3$ as $n\to\infty$. By contrast, we
will see that for $m=4,5$, both functions $\mu(n,m)$ and $\mu^*(n,m)$ are asymptotic to
$4n/7$. In fact we have the following result:
\begin{thm}\label{T:m=4} Let $n\in\N$.
\begin{enumerate}
\item For $j=4,5$, $\mu(n,j)$ is given by the equation
$$
\mu(n,j) =
  \begin{cases}
  \setlength{\jot}{55pt}
  \lce\ds{\frac{4n}7}\rce, & n \ne 7k-2 \text{ for some } k\in\N, \\[12pt]
  \lce\ds{\frac{4n}7}\rce+1=4k, & n=7k-2 \text{ for some } k\in\N. \\[.5em]
  \end{cases}
$$
\item If $n\ge 3$, then $\mu^*(n,4)=2\lce\ds{\frac{2n}7}\rce$ is the
    least even integer not less than $\mu(n,4)$.
\end{enumerate}
\end{thm}

It is not hard to show that $\mu(n,m)\ge n/2$ for all $n,m\in\N$. This is
asymptotically sharp according to the following result.

\begin{thm}\label{T:lim}
$\lim\limits_{n\to\infty} \mu(n)/n=1/2$.
\end{thm}

In fact, this shows that $\displaystyle \frac{\mu(n)}{\nu(n)} \to 1$
(cf. Theorem C below).   

\subsection{Outline}
The article is organized as follows. After the some introductory material in
Section \ref{S:review}, we consider $\mu(n,m)$ and
$\mu^*(n,m)$ for numbers of the form $(n,m)=(2^{k+1}-1,2^k)$ in
\rf{S:m=2^k}. This special case involves a wiring related to Hadamard
matrices, and allows us to deduce \rf{T:lim}.

In \rf{S:U}, we give an explicit upper bound $U(n,m)$ for $\mu(n,m)$. This
upper bound has the appearance of being rather sharp: indeed, we know of no
pair $(n,m)$ such that $\mu(n,m)<U(n,m)$. Whether $\mu(n,m)=U(n,m)$ for all $n,m$
is an interesting open question. The upper bound $U(n,m)$ sheds light on the
formulae for $\mu(n,m)$ given above and in \rf{S:review} which, although
convenient for understanding the asymptotics of $\mu(n,m)$ as $n\to\infty$,
do not seem to follow any clear pattern as $m$ changes.
The sequence $U(n,n)$ is connected to OEIS sequence A046699, which
is of meta-Fibonacci type, and has a number of combinatorial 
descriptions in terms of trees.

In \rf{S:m near 2^k}, we prove that if $\mu(\cdot,m)=U(\cdot,m)$ for
$m=2^k-2$, then this equation also holds for $m=2^k+i$, $i\in\{-1,0,1\}$.
\rf{T:m=4}(a) will follow immediately from this result but \rf{T:m=4}(b)
still requires a proof, which can be found in \rf{S:m=4}.

\section{A recap of previous results and ideas} \label{S:review}\label{S:notation}\label{S:recap}

For ease of reference, we state and label some
results from \cite{BOF}. We need them either for
proofs or for comparison purposes. 

\subsection{Theorems from \cite{BOF}}
We begin by listing the three main results in \cite{BOF}: in the order
listed below, these were Theorems 1.1, 1.2, and 3.2 in that article.

\begin{Thm}\label{T:m=2} Let $n\in\N$.
\begin{enumerate}
\item $\mu(n,2) = \lce\ds{\frac{2n}3}\rce$.
\item If $n\ge 2$, then $\mu^*(n,2) = 2\lce\ds{\frac{n}3}\rce$ is the
    least even integer not less than $\mu(n,2)$.
\end{enumerate}
\end{Thm}

\begin{Thm}\label{T:m=3} Let $n\in\N$.
\begin{enumerate}
\item $\mu(n,3) = \mu(n,2)$.
\item If $n\ge 3$, then
$$
\mu^*(n,3) =
  \begin{cases}
  4k-1,     & n = 6k-3 \text{ for some } k\in\N, \\
  \mu(n,3), & \text{otherwise}.
  \end{cases}
$$
\end{enumerate}
\end{Thm}
Note that $\mu^*(n,3) = \mu(n,3)+1$ in the exceptional case $n=6k-3$.

\begin{Thm}\label{T:nu}
Let $n,m\in\N$, $m>1$.
\begin{enumerate}
\item $\nu(n) = \nu(n,m) = \lce \ds{\frac{n}2} \rce$.
\item If $n\ge m$, then
$$
\nu^*(n,m) =
  \begin{cases}
  \nu(n,m)+1, &\text{if $\/n$ is even and $\/m$ odd}, \\
  \nu(n,m),   &\text{otherwise}.
  \end{cases}
$$
In particular, $\nu^*(n,2)=\nu^*(n)=\nu(n)$ for all $n>1$.
\end{enumerate}
\end{Thm}


\subsection{Lemmas from \cite{BOF}}
The next four results were, in the order listed below, Lemmas 3.1, 5.1, and
5.2, and Corollary 3.3 in \cite{BOF}.

\begin{Lem}\label{L:mean}
Let $n\in\N$. For all $W\in A(n)$ and $c\in\F_2^n$, the mean value of
$|Mx+c|$ over all $x\in\F_2^n$ is $n/2$. In particular, $M(W,c)\ge n/2$ and
$M(W,c)>n/2$ if the cardinality of $\{i\in[1,n]\cap\N\mid c_i=1\}$ is not
$n/2$.
\end{Lem}

\begin{Lem}\label{L:alternative}
Let $m\ge 2$ and $n\ge 1$. Then either $\mu(n+m,m)=\mu(n+m,m-1)$, or
$$ \mu(n+m,m)\ge \mu(n,m)+\nu(m) = \mu(n,m) + \lce {\frac{m}2} \rce\,. $$
\end{Lem}

\begin{Lem}\label{L:nn'm}\label{L:F}
Let $n,m,n'\in\N$, with $n\ge m$. Then
$$ \mu^*(n+n',m+1)\le \mu^*(n,m)+n'\,. $$
\end{Lem}

\begin{Cor} \label{C:sublinear}
If $\la$ is any one of the four functions $\mu$, $\mu^*$, $\nu$, or $\nu^*$,
then $\la(\cdot,m)$ is sublinear for all $m$:
\begin{equation}\label{E:sublinear}
\la(n_1+n_2,m) \le \la(n_1,m) + \la(n_2,m)\,,
\end{equation}
as long as this equation makes sense (i.e.~we need $n_1,n_2\ge m$ if
$\la=\mu^*$ or $\la=\nu^*$).
\end{Cor}

\subsection{Edge functions}
Associated with the graph $G$ is its vertex set $S$ (which we treat as 
an initial segment $S(n):=\{1,\ldots,n\}$ of
the set $\N$ of natural numbers) and the
{\it edge function} $F:S\to 2^S$, where $j\in F(i)$ if there is
an edge from $i$ to $j$, and the {\it backward edge function} $F^{-1}:S\to
2^S$, where $j\in F^{-1}(i)$ if there is an edge from $j$ to
$i$. We extend the definitions of $F$ and $F^{-1}$ to $2^S$ in the usual
way: $F(T)$ and $F^{-1}(T)$ are the unions of $F(i)$ or $F^{-1}(i)$,
respectively, over all $i\in T\subset S$. We say that $T\subset S$ is {\it
forward invariant} if $F(T)\subset T$, or {\it backward invariant} if
$F^{-1}(T)\subset T$. Given a wiring $W$, associated graph $G$, and
$T\subset S$, we denote by $W_T$ and $G_T$ the subwiring and subgraph,
respectively, associated with the vertices in $T$: more precisely, $W_T$ is
the matrix obtained by deleting all rows and columns of $W$ other than those
with index in $T$, and $G_T$ is obtained by retaining only the vertices in
$T$ and those edges in $G$ between vertices in $T$.

\subsection{Pivoting} We now recall the concept of {\it pivoting}, as introduced in
\cite[Section~5]{BOF}. Pivoting about a vertex $i$, $1\le i\le n$, is a way
of changing the given wiring $W$ to a special wiring $W^i$ such that
$M(W^i,c)\le M(W,c)$. Additionally, pivoting preserves the classes $A(n,m)$
and $A^*(n,m)$.

Let us fix a wiring $W=(w_{i,j})$ on $n$ vertices, and let $F:S\to 2^S$
denote the edge function associated to $W$, where $S=S(n)$. Given
$T\subset S$, and $i\in S$, we define $W^{i,T}$ by replacing the $j$th
column of $W$ by its $i$th column whenever $j\in F(i)\setminus T$. We refer
to the wiring $W^{i,T}$ as the {\it pivot of $W$ about $i$ relative to $T$}.
If $T$ is nonempty, we refer to this process as {\it partial pivoting},
while if $T$ is empty we call it {\it (full) pivoting} and 
write $G^i$, $W^i$, and $F^i$ for the resulting graph, matrix,
and edge function, respectively.

As in \cite{BOF}, we use the notation $\hat K_r$ to denote
an augmented complete graph on $r$ vertices, i.e.
a complete graph augmented by a loop at each
vertex.  Full pivoting about vertex $i$ just rewires $F(i)$
so that it becomes a $\hat K_{\deg(i)}$, which is 
thus a forward-invariant subgraph of $W^i$.

We refer to a forward-invariant $\hat K_r$ subgraph of a wiring graph
$W$ as an $F_r$ (relative to $W$).

For $t\in\{0,1\}$, we denote by $t_{p\times q}$ the $p\times q$ matrix all
of whose entries equal $t$, and let $t_p=t_{p\times p}$. 
The matrix of a  $\hat K_r$, is $1_{r\times r}$.
This is (of course) different from the $r\times r$ identity matrix $I_r$,
except when $r=1$.

Pivoting relative to any $T$ is a process with several nice properties: it
has the non-increasing property $M(W^{i,T},c)\le M(W,c)$, it preserves
membership of the classes $A(n,m)$ and $A^*(n,m)$, and if $F^{i,T}$ is the
edge function of $W^{i,T}$, then $F^{i,T}(i)=F(i)$ is an augmented complete 
subgraph of
the associated graph $G^{i,T}$, but might not be forward invariant
in $G^{i,T}$. 

\subsection{Graphical conventions} We continue the graphical conventions
introduced in \cite{BOF}.
Thus, we do not show loops or
the internal edges in a $\hat K_r$, and a single arrow issuing from
$\hat K_r$ represents $r$ edges, one from each vertex in the
$\hat K_r$, all sharing the same target. If several arrows  
from a $\hat K_r$ point to some $\hat K_s$, then distinct
arrows have distinct targets (so the number of arrows will not
exceed $s$). For instance, Figure \ref{F:2-K-6-3} shows 
three views of a $\hat K_6$.
\begin{figure}[h]
        \begin{center}
                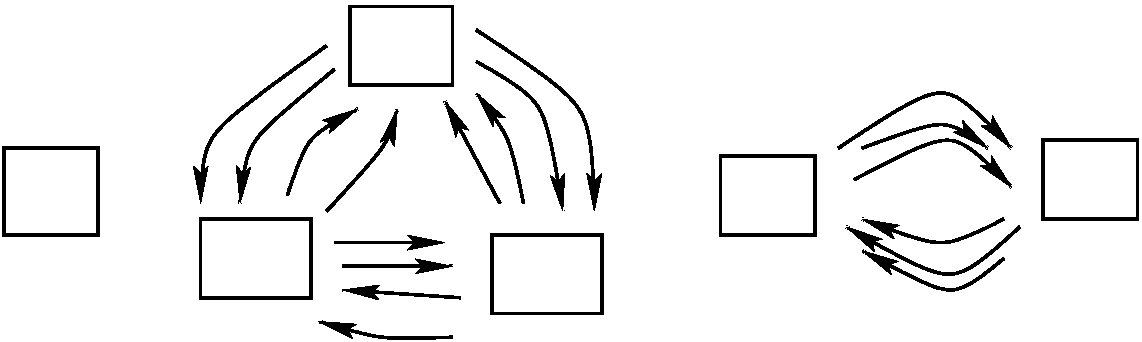
		\caption{Views of $\hat K_6$}\label{F:2-K-6-3}
        \end{center}
\end{figure}
Notice how the $36$ directed edges of the $\hat K_6$
are hidden to varying degrees in this figure, and how the
arrows represent multiple edges --- two each in the 
version with $\hat K_2$s, and three each in the version
with $\hat K_3$s. To reduce clutter
further, we introduce the additional convention that
an two-headed arc stands for a pair of arrows, one
in each direction.  This gives the two more views
of $\hat K_6$ shown in Figure \ref{F:2-K-6-2}
\begin{figure}
	\begin{center}[h]
                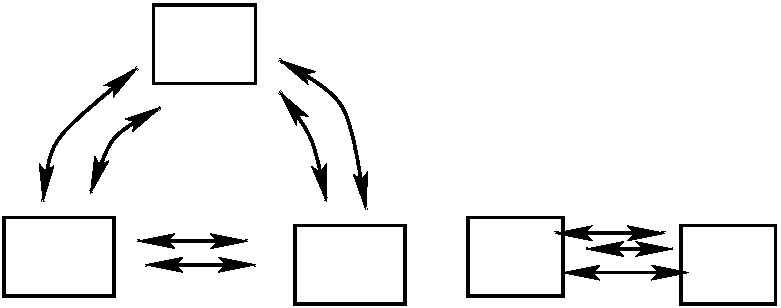
		\caption{More views of $\hat K_6$}\label{F:2-K-6-2}
        \end{center}
\end{figure}
in which individual two-headed arcs stand for up to six
directed edges in the $\hat K_6$.

\section{The case $(n,m)=(2^{k+1}-1,2^k)$} \label{S:m=2^k}

\subsection{}
We begin with some observations for general $n,m$ that will be useful here
or in later sections. Trivially, $\mu(n,m)$ is nonincreasing as a function
of $m$, but it is also easy to see that it is also nondecreasing as a
function of $n$: given a wiring $W\in A(n,m)$ such that $|Wx|\le\mu(n,m)$
for all $x\in\F_2^n$, it may be that vertex $n$ has degree $1$, in which
case it is clear that if $W'$ is obtained by eliminating the last row and
column of $W$, then $|W'x'|\le\mu(n,m)$ for all $x'\in\F_2^{n-1}$.

If instead vertex $n$ has degree larger than $1$ then, by pivoting if
necessary, we may assume that vertex $n$ forms a part of a forward invariant
$\hat K_j$ for some $j>1$. Because the effect of pressing vertex $n$ is the same
as the effect of pressing any other vertex in the $\hat K_j$, the set of vectors
$Wx$, as $x=(x_1,\dots,x_n)^t$ ranges over all vectors in $\F_2^n$ for which
$x_n=0$, coincides with the set of vectors $Wx$ as $x$ ranges over all of
$\F_2^n$. It follows that if we define $W'$ as in the previous case, then
$|W'x'|\le\mu(n,m)$ for all $x'\in\F_2^{n-1}$.

\textbf{In contrast, we do not know whether or not $\mu^*(n,m)$ is an nondecreasing
function of $n$.}

\subsection{}
Another easily proven inequality is:
\begin{equation}\label{E:+1}
\mu(n+1,m)\le \mu(n,m)+1\,.
\end{equation}
To see this, we need only consider the matrix $W\in A(n+1,m)$ which has
block diagonal form $\diag(W',I_1)$, where $W'\in A(n,m)$ satisfies
$M(W',0)=\mu(n,m)$.

\subsection{}
We now prove a pair of closely related lemmas. We will only use the second
one in this section, but we will need the first one later.

\begin{lem}\label{L:reflect}
Let $m,m',n\in\N$ and $m\le n$. Then
\begin{align*}
\mu(nm',mm')&\le m'\mu(n,m) \\
\mu^*(nm',mm')&\le m'\mu^*(n,m)
\end{align*}
\end{lem}

\begin{proof}
Essentially the same proof works for $\mu$ and $\mu^*$, so we write down
only the one for $\mu$. Let $W\in A(n,m)$ be such that $M(W,0)=\mu(n,m)$. We
construct a new matrix $W'$ by replacing each entry $w_{i,j}$ in $W$ by an
$m'\times m'$ block, each of whose entries is $w_{i,j}$, i.e.~$W'$ is the
Kronecker product $W\bigotimes 1_{m'\times m'}$.
It is readily
verified that $W\in A(nm',mm')$.

The graph of $W$ is obtained by replacing each vertex $j$ in the original graph
$G$ by $m'$ new vertices which we will label $(j,j')$, $1\le j'\le m'$. Pressing
vertex $(j,j')$ changes the status of some other vertex $(i,i')$ if and only
if pressing $j$ changes the status of vertex $i$ in the original graph. 
In the new wiring $W'$, each bulb of $W$ has been replaced by a bank of
$m'$ bulbs, all of which are switched synchronously by
any of their associated switches and
it is clear that $M(W',0)=m'M(W,0)$.
\end{proof}

Figure \ref{F:2-1} illustrates the proof
that $\mu^*(18,9)\le 3\mu^*(6,3) (=12)$, i.e.~the
case $n=6$, $m=3$, $m'=3$.  
\begin{figure}[h]
	\begin{center}
		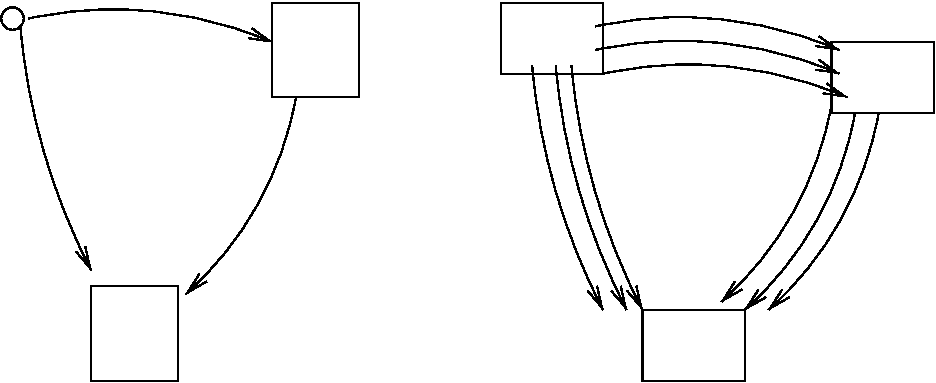
		\caption{$W$ and $W'$}\label{F:2-1}
	\end{center}
\end{figure}
The graph $W$ is the graph from Figure 
12 in \cite{BOF},
the wiring example which concludes the proof
that $\mu^*(6,3)=4$. 
To construct $W'$, each vertex of $W$
has been replaced by a $\hat K_3$ and each edge by three 
edges, one to each vertex of the $\hat K_3$ 
that replaces the original target. Thus, the original
$\hat K_2$ and $\hat K_3$ become a $\hat K_6$
and a $\hat K_9$, respectively. In terms of bulbs and switches,
each bulb becomes a bank of $3$ bulbs, and each
switch a bank of $3$ switches, all having the same effect.

It is also possible to view the new wiring 
$W'$ as a row of $m'$ copies
of $W$, suitably wired together, and when we 
think of it in this way we refer to the copies as
\emph{clones} of $W$. Figure \ref{F:2-2} shows 
this view of the above example. The view
in Figure \ref{F:2-2} is comparatively cluttered,  
but it is still substantially less messy than the full wiring
graph, which has $162$ directed edges.
\begin{figure}
	\begin{center}
		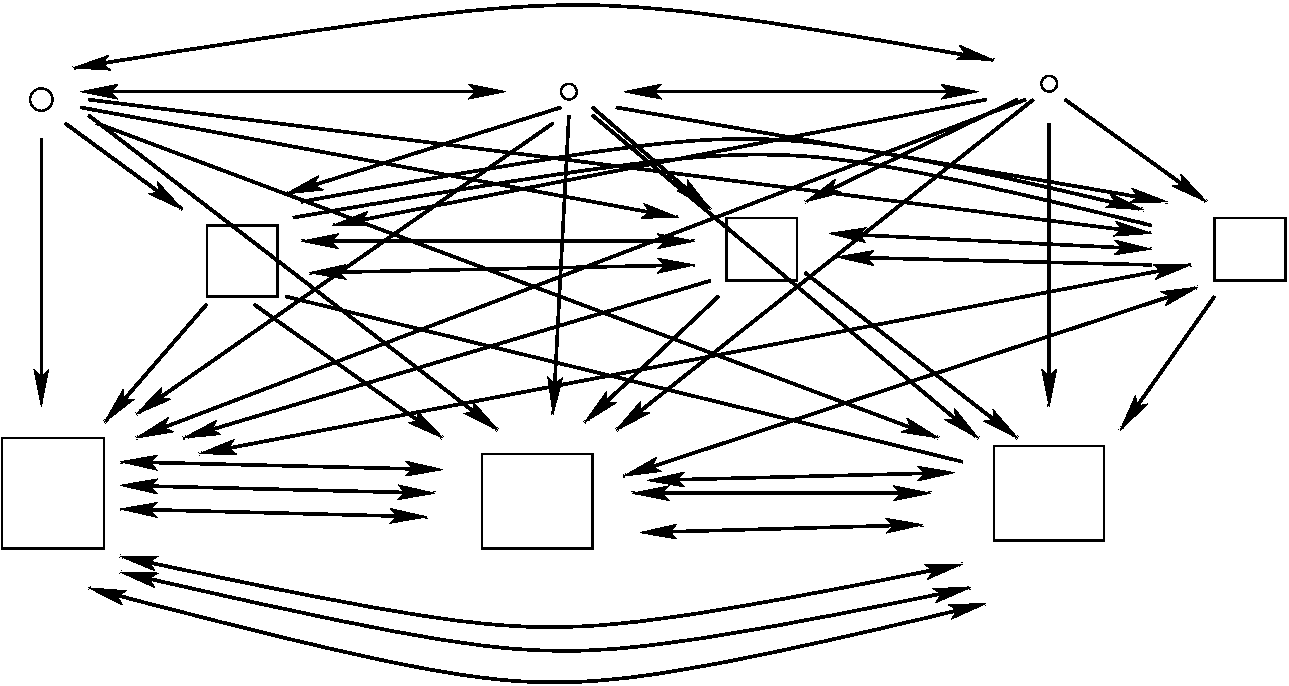
		\caption{$W'$}\label{F:2-2}
	\end{center}
\end{figure}

\begin{lem}\label{L:2n+1}
Let $n,m,m'\in\N$ with $m'm\ge n+1$. Then $\mu(m'n+1,m'm)\le m'\mu(n,m)$.
\end{lem}

\begin{proof}
Let $W\in A(n,m)$ be such that $M(W,0)=\mu(n,m)$. As in the previous lemma,
we construct a new matrix $W'=W\bigotimes 1_{m'\times m'}$. The wiring $W'$
is a wiring for $m'n$ vertices which can be split into $n$ banks of $m'$
vertices that are always in sync (either all on or all off). We add one
last vertex $v$ and get a new wiring by connecting $v$ to itself and to one
vertex from each of the $m'$ sets of clones. In terms of matrices, this
can be achieved by defining a matrix with block form
\begin{equation} \label{E:W''}
W'' =
  \begin{pmatrix}
  W' & V \\ 0_{1\times m'n} & I_1
  \end{pmatrix}
\end{equation}
where $V=(v_i)$ is a $m'n\times1$ column vector with $v_i=1$ if $i$ is a
multiple of $m'$, and $v_i=0$ otherwise.
Using the inequality $m'm\ge n+1$, it is readily verified
that $W\in A(m'n+1,m'm)$.

If we do not press $v$, then it is clear (as in the previous proof) that we
can light at most $m'M(W,0)=m'\mu(n,m)$. Suppose therefore that we press $v$
(together with some combination of other vertices). Partitioning each set of
$m'$ clones into two subsets $S'$ and $S''$, where $S''$ has
cardinality $2$ and includes the vertex which is toggled by $v$, it is clear
that all vertices in each of the $S'$ sets remain in sync, that precisely
one vertex in each $S''$ is lit, and that $v$ itself is lit. Thus, we can
light at most $(m'-2)\mu(n,m)+n+1$ if $v$ is pressed. Since we know from
\rf{L:mean} that $\mu(n,m)=M(W,0)>n/2$, we have $n+1\le 2\mu(n,m)$, and so
$(m'-2)\mu(n,m)+n+1\le m'\mu(n,m)$, and we are done.
\end{proof}

\subsection{}
We now state our first main result for $m$ close to a power of $2$.

\begin{thm}\label{T:m=2^k} For all $k\in\N$, and $m\ge 2^k$,
$$ \mu(2^{k+1}-1,m) = \mu^*(2^{k+1}-1,2^k) = 2^k\,. $$
\end{thm}

\subsection{}
Using this theorem, it is easy to deduce \rf{T:lim}, i.e.~$\lim_{n\to\infty}
\mu(n)/n= 1/2$:

\begin{proof}[Proof of \rf{T:lim}]
\rf{L:mean} implies that $\liminf_{n\to\infty} \mu(n)/n\ge 1/2$, so it
suffices to show that $\limsup_{n\to\infty} \mu(n)/n\le 1/2$. Fixing
$k\in\N$, let us assume that $n>p:=2^{k+1}-1$. We write $n=ap+r$, where
$a\in\N$ and $0\le r\le p-1$. By inequality \rf{E:sublinear}, \rf{T:m=2^k}, and the
fact that $\mu(\cdot,\cdot)$ is nondecreasing in its first argument and
nonincreasing in its second, we see that
$$
\mu(n) \le \mu(n,2^k) \le a\mu(p,2^k)+\mu(r,2^k) \le (a+1)2^k\,.
$$
Letting $n\to\infty$, it follows easily that
$\limsup_{n\to\infty}\mu(n)/n\le 2^k/(2^{k+1}-1)$. Since $k$ can be chosen
to be arbitrarily large, it follows that $\limsup_{n\to\infty} \mu(n)/n\le
1/2$, as required.
\end{proof}

\subsection{Sylvester-Hadamard matrices}
Before proving \rf{T:m=2^k}, we need to discuss the Sylvester-Hadamard
matrices, which are defined as follows:
$$
H_2 =
  \begin{pmatrix}
  1 &  1 \\
  1 & -1
  \end{pmatrix}
$$
and inductively $H_{2^k}$ is given in block form by
$$
H_{2^k} =
  \begin{pmatrix}
  H_{2^{k-1}} &  H_{2^{k-1}} \\
  H_{2^{k-1}} & -H_{2^{k-1}}
  \end{pmatrix}\,.
$$
Equivalently, $H_{2^k}$ is the Kronecker product $H_2\bigotimes H_{2^{k-1}}$.

Let $h$ be the rescaled Haar function given by $h(t)=1$ if $\lfl t \rfl$ is
even and $h(t)=-1$ otherwise. Let $h_p(t)=h(2^{-p}t)$ for all $p\in\N$,
so that each function $h_p$ is periodic. It is straightforward to verify
that for fixed $k\in\N$ and $1\le j\le 2^k$, the $j$th column
$(a_{i,j})_{i=1}^{2^k}$ of $H_{2^k}$ is always given by a pointwise product
of one or more of the column vectors $(h_p(i-1))_{i=1}^{2^k}$, $1\le p\le
k$, and that any such product gives some column of $H_{2^k}$. It follows
that a pointwise product of any number of the columns of $H_{2^k}$ is
another column of $H_{2^k}$, a fact that will be useful in the following
proof.

\subsection{Proof of \rf{T:m=2^k}}
\begin{proof}
By \rf{L:mean} and the fact that $\mu(\cdot,\cdot)$ is nonincreasing in its
second argument, we have that $\mu^*(2^{k+1}-1,2^k)\ge\mu(2^{k+1}-1,m)\ge
2^k$. Conversely, by taking $n=2^{j+1}-1$, $m=2^j$, and $m'=2$ in
\rf{L:2n+1}, we deduce inductively $\mu(2^{k+1}-1,2^k)\le 2^k$, and so
$\mu(2^{k+1}-1,m)\le 2^k$.

It remains to get the same upper bound for $\mu^*(2^{k+1}-1,2^k)$. For this,
	we need to work a little harder. Fix $k$ and let $n=2^{k+1}-1$.
	We claim that if we delete the first row
and column of the Sylvester-Hadamard matrix $H_{2^{k+1}}$, and change each
$1$ entry to a $0$ and each $-1$ to a $1$, then we get an $n\times n$
	matrix $W=W_k$ over $\F_2$ such that each column of $W$
	has exactly $2^k$ ones, and such that the pointwise
	sum of any two columns of $W$ is another column of $W$ or is
	a column of zeros.

The fact that any
pointwise product of columns of $H_{2^{k+1}}$ is another column of that same
matrix means that any pointwise product of columns of $H_{2^{k+1}}$ has
either zero or $2^k$ entries equal to $-1$. Pointwise products for
$H_{2^{k+1}}$ correspond to pointwise sums mod $2$ for $W$, so the claim
	is established.

	It follows from the claim that for each $x\in\F_2^n$, 
	the vector $Wx$ is some column of $W$, so $|Wx|=2^k$ or $0$.

For $k\in\N$, the graph with matrix $W_k$ does not have a loop at each vertex,
i.e. it corresponds to an inadmissible wiring.
But if $V$ is any matrix all of whose columns are columns of $W_k$,
and which has only $1$'s on the diagonal, then 
$V\in A^*(n,2^k)$ and for each $x\in\F_2^n$
we have $|Vx|=2^k$ or $0$ (because $Vx=WPx$ for some projection $P$), 
so we deduce that 
$M(V,0)=2^k$.  The simplest way to construct such a matrix $V$
from $W$ is to repeat columns $1$,$2$,$4$ and so on, respectively,
once, twice, four times, etc. In other words, take column $i$ of
	$V$ equal to column $2^j$ of $V$ whenever $2^j\le i< 2^{j+1}$.
This concludes the proof.
\end{proof}

\subsection{Towers $V_k$}\label{SS:towers}
The matrix $V=V_k$ in the foregoing proof has the property that
the nonzero entries occur in blocks that are of the form
$1_{r\times r}$, where $r$ runs through powers of $2$.
Graphically, this wiring $V$ corresponds to a tower of
$k+1$ augmented complete graphs, one of degree equal to each 
power of $2$, as illustrated (sideways on) in Figure \ref{F:2-5}.
\begin{figure}
	\begin{center}
		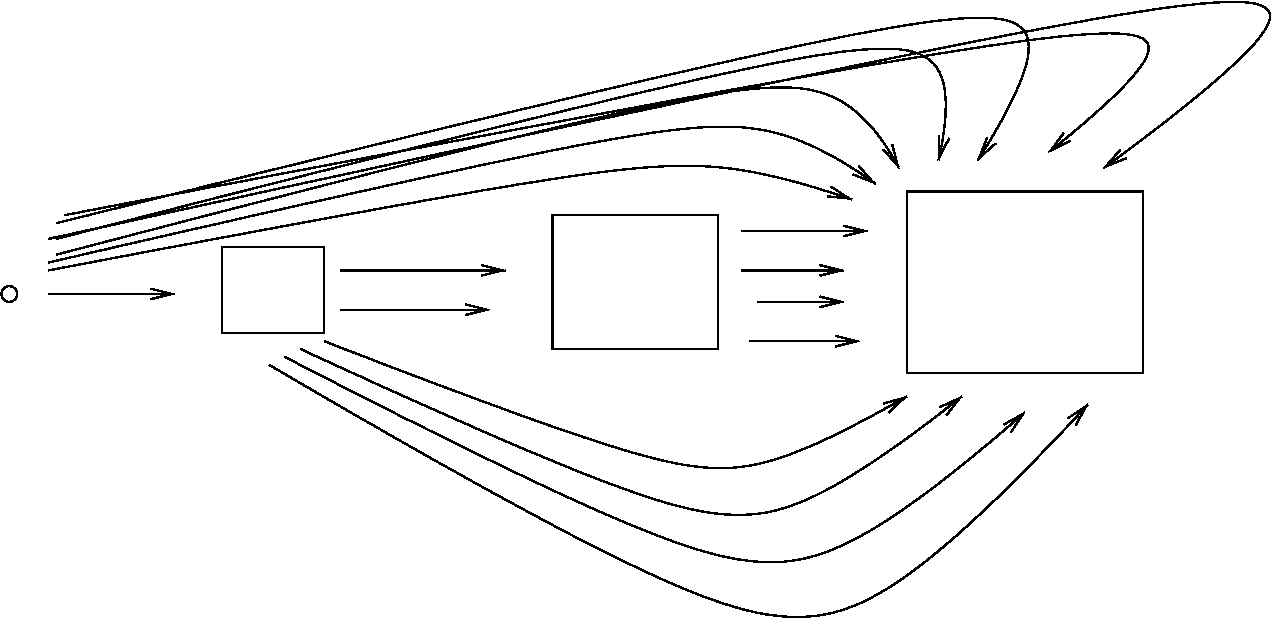
		\caption{$V_3$}\label{F:2-5}
	\end{center}
\end{figure}
The corresponding matrix is
$$\left(\begin{array}{rrrrrrrrrrrrrrr}
	1&0&0&0&0&0&0 &0&0&0&0&0&0&0&0
	\\
	0&1&1&0&0&0&0    &0&0&0&0&0&0&0&0
	\\
	1&1&1&0&0&0&0    &0&0&0&0&0&0&0&0
	\\
	0&0&0&1&1&1&1     &0&0&0&0&0&0&0&0
	\\
	1&0&0&1&1&1&1     &0&0&0&0&0&0&0&0
	\\
	0&1&1&1&1&1&1     &0&0&0&0&0&0&0&0
	\\
	1&1&1&1&1&1&1    &0&0&0&0&0&0&0&0
	\\
	0&0&0&0&0&0&0 &1&1&1&1&1&1&1&1
	\\
	1&0&0&0&0&0&0 &1&1&1&1&1&1&1&1
	\\
	0&1&1&0&0&0&0 &1&1&1&1&1&1&1&1
	\\
	1&1&1&0&0&0&0&1&1&1&1&1&1&1&1
	\\
	0&0&0&1&1&1&1 &1&1&1&1&1&1&1&1
	\\
	1&0&0&1&1&1&1&1&1&1&1&1&1&1&1
	\\
	0&1&1&1&1&1&1 &1&1&1&1&1&1&1&1
	\\
	1&1&1&1&1&1&1&1&1&1&1&1&1&1&1
\end{array} \right)
$$
	
The transition from $W_k$ to $V_k$ in the proof can be described by
a sequence of pivots: A first pivot produces a forward-invariant $\hat K_{2^k}$,
then a partial pivot with respect to the $\Hat K_{2^k}$
produces a $\hat H_{2^{k-1}}$, and so on.  The process converts
a rather symmetrical inadmissible graph into an asymmetric admissible
tower.  Alternative constructions that amount to multiplying
$W_k$ by a permutation matrix convert the inadmissible graph
to a symmetric admissible graph without
a proper forward-invariant subgraph. Figure \ref{F:2-6} 
shows an example, obtained by permuting the columns of
$V_2$ to the order $(1,2,5,6,3,4,7)$ (As usual, the loops
at the vertices are not shown.) This could be illustrated
rather prettily on a regular tetrahedron by placing $1$
at the apex, the $2,4,6$ as the vertices of the base
triangle, and placing the remaining three points on the 
edges halfway up, with $5$ on the edge $1-2$, $7$ on $1-4$,
and $3$ on $1-6$.
All the arrows can then be drawn on 
faces of the tetrahedron.
\begin{figure}[h]
	\begin{center}
		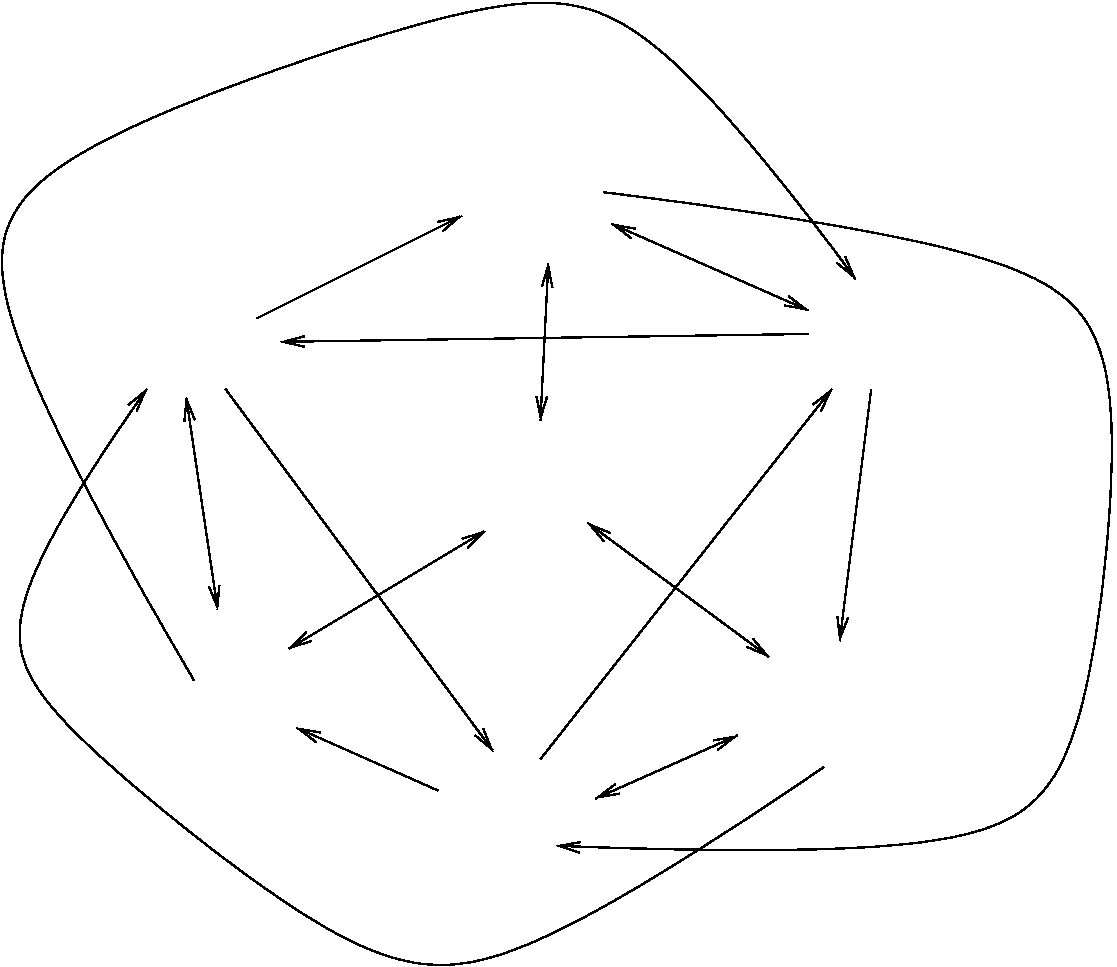
		\caption{An alternative $V$ for $k=2$}\label{F:2-6}
	\end{center}
\end{figure}

\subsection{Remark}
Note that there are Hadamard matrices $H_{4n}$ of dimension $4n$ for many
$n\in\N$, not just powers of $2$; in fact, they are conjectured to exist for
all dimensions $4n$ \cite{Sloane-Hadamard}. 
Since by definition the rows of an Hadamard matrix are
pairwise orthogonal, one might wish to use $H_{4n}^t$ as we used the
symmetric Sylvester-Hadamard matrices. However, this is not possible for
several reasons: we do not in general have a complete row and column of $1$s
suitable for deleting (although there is always an equivalent Hadamard
matrix with this property), there may not be $-1$s along the diagonal of an
associated minor, and some pointwise products of more than two columns of
$H_{4n}$ may have more than $2n$ entries equal to $-1$ (even if $n$ is a
power of $2$). For instance, in the Paley-Hadamard matrix
$$
P =
  \(\begin{array}{rrrrrrrr}
  1 & 1 & 1 & 1 & 1 & 1 & 1 & 1 \\-1 & 1 &-1 &-1 & 1 &-1 & 1 & 1 \\
 -1 & 1 & 1 &-1 &-1 & 1 &-1 & 1 \\-1 & 1 & 1 & 1 &-1 &-1 & 1 &-1 \\
 -1 &-1 & 1 & 1 & 1 &-1 &-1 & 1 \\-1 & 1 &-1 & 1 & 1 & 1 &-1 &-1 \\
 -1 &-1 & 1 &-1 & 1 & 1 & 1 &-1 \\-1 &-1 &-1 & 1 &-1 & 1 & 1 & 1 \\
  \end{array}\)\,,
$$
the pointwise product of columns 2, 3, and 5 contains all $-1$s, except
from the first entry. For all
these reasons, the method for Sylvester-Hadamard matrices does not in other
cases produce a $W\in A(4n-1,2n)$, let alone $W$ such that $M(W,0)=2n$.

\subsection{Codes}
Hadamard matrices generate Hadamard codes, which have a certain optimality
property. Recall that the code associated with $H^{2^k}$ has $2^{k+1}$
codewords that make up a group $G<(\F_2^{2^k},+)$. The above wiring
$W\in A^*(2^k-1,2^{k-1})$ can by constructed from the code as follows.
First, let $H<G$ be the order 2 subgroup generated by $(1,\dots,1)^t\in G$, and
select the element in each coset of $H$, other than $H$ itself, that has a
$0$ in the first coordinate. Discarding the first coordinate of each
selected codeword yields a set of projected codewords that give the columns
of $W$.

It would be interesting to know if there are any further connections between
optimal codes and optimal wirings. There is a reason to expect that
(near-)optimal linear codes may be associated with (near-)optimal
wirings: a near-optimal linear code is one in which the minimum over all
codewords $w$ of the Hamming distance $|w|$ is about as large as possible, so if
we take many of these codewords as the columns of the wiring matrix (perhaps
after discarding one or more coordinates, as we did for Hadamard codes), we
get a matrix for which $|Wx|$ is fairly large, except for the relatively few
times when $Wx=0$. A relatively large minimum nonzero value for $|Wx|$
should therefore be associated with a relatively small maximum value for
$|Wx|$, since \rf{L:mean} says that the average of $|Wx|$ over all
$x\in\F_2^n$ is $n/2$.

\section{An upper bound for $\mu(n,m)$} \label{S:U}
In this section, we establish an upper bound $U(n,m)$ for $\mu(n,m)$ in all
cases. This upper bound seems rather sharp, insofar as we know of no values
$n,m$ for which $\mu(n,m)$ and $U(n,m)$ differ. We also investigate $U(n,m)$
and a related nondecreasing sequence $(a(n))_{n=1}^\infty$ which we use to
define $U$.

\subsection{The sequence $a(n)$}
We first define $(a(n))$ by the following inductive process:
\begin{align*}
a(1) &= 1\,, \\
a(2^k-1+i) &= 2^{k-1}+a(i)\,, \qquad &1\le i\le 2^k-1,\;k\in\N\,, \\
a(2^{k+1}-1) &= 2^k\,, \qquad &k\in\N\,.
\end{align*}
Thus $(a(n))$ begins:
\begin{align*}
1,\;&2,\;2,\;3,\;4,\;4,\;4,\;5,\;6,\;6,\;7,\;8,\;8,\;8,\;8,\;9,\;10,\;10,\;11,\;
12,\;12,\;12,\;13,\;14,\;14,\;14,\\
&\;15,\;16,\;16,\;16,\;16,\;16,\;17,\dots
\end{align*}
It is not hard to verify that the above sequence has the following
alternative description: it is the nondecreasing sequence consisting of all
positive integers, where the frequency of each integer $n$ is the $2$-adic
norm of $2n$.

Note that $a(n)\le 2^k$ whenever $n\le 2^{k+1}-1$.

If we add an extra $1$ term to the beginning of the sequence $(a(n))$, we
get a sequence $(b(n))$ listed in the OEIS (Online Encyclopedia of Integer
Sequences) as A046699 \cite{OEIS}. The sequence $(b_n)$ is defined by the
initial conditions $b(1)=b(2)=1$, and the following recurrence relation:
$$ b(n) = b(n - b(n-1)) + b(n-1 - b(n-2))\,, \qquad{n>2}\,. $$
It can be deduced from this that $(a(n))$ satisfies the same recurrence
relation as $(b(n))$: we just need to modify the initial conditions. We
leave the verification of this to the reader, with the hint that it is
straightforward to deduce it from the two inequalities $a(n)>n/2$ and
$a(n+1)\leq a(n)+1$.

Such so-called {\it meta-Fibonacci sequences} go back to D.~Hofstadter
\cite[p.~137]{H}, and are generally considered to be rather mysterious.
Indeed, one of them was the subject of a \$10\,000 prize offered by the 
late J.~Conway
\cite{C2}. However $(a(n))$ and $(b(n))$ are clearly rather tame members of
this family, and one or other has appeared elsewhere in the context of
binary trees; see \cite{JR}, \cite{DR}, \cite{C1}, and \cite{E}.

\subsection{The function $U$}
Having defined $(a(n))$, we are now ready to define our upper bound function
$U$. Given positive integers $n$ and $m$, we choose the nonnegative
integer $k$ for which $2^k\le m<2^{k+1}$,
and we let $q$ and $r$ be the integers with
$n=(2^{k+1}-1)q+r$, with $q\ge 0$ 
and $1\le r<2^{k+1}$. Thus,  
$q$ and $r$ are the usual integral quotient and remainder when
$n$ is divided by $2^{k+1}-1$, except when the remainder is zero,
and in that case $r=2^{k+1}-1$ and $q=(n-r)/(2^{k+1}-1)$. 
We then define $U(n,m)=q2^k+a(r)$.

\ignore{
	AOF: EXPLANATION written out in my notes. 
}

Note that given $n$ and $m$, the integers $k,r,q$ so defined are unique,
and that $U(n,m)\ge a(n)$. Following our usual notation, we write
$U(n)=U(n,n)$, so that $U(n)$ is just an alternative notation for $a(n)$.

\begin{prop}\label{P:general m} We have $\mu(n,m) \le U(n,m)$ for all
$n,m\in\N$.
\end{prop}

\begin{proof}
The result is trivially true when $m=1$. We prove the result for $2^k\le
m<2^{k+1}$ by induction on $k$. Assuming $\mu(\cdot,m) \le U(\cdot,m)$ for
$2^{k-1}\le m<2^k$, we need to prove that this estimate also holds for
$2^k\le m < 2^{k+1}$. From now on, we assume that $2^k\le m<2^{k+1}-1$.

Sublinearity of $\mu(\cdot,m)$ and the inductive hypothesis gives
$\mu(n,2^k-1)\le 2^k$ for all $n<2^{k+1}-1$. Since any wiring with a vertex
of degree at least $2^k$ allows us to light at least $2^k$ vertices, we must
have
$$ \mu(n,m) = \mu(n,2^k-1) \le U(n,2^k-1) = U(n,m) $$
for $n<2^{k+1}-1$. If $n=2^{k+1}-1$, then \rf{T:m=2^k} yields
$$ \mu(n,m) = \mu(n,2^k) = 2^k = U(n,m)\,.$$
Finally, the required inequality follows readily for all $n\ge 2^{k+1}$ by
using the case $n<2^{k+1}$ and sublinearity of $\mu(\cdot,m)$. Thus, we have
proven the inductive step, and we are done.
\end{proof}

\subsection{Sublinearity}
Equation \rf{E:sublinear} says that $\mu(\cdot,m)$ is sublinear for all $m$.
We now prove the same for $U(\cdot,m)$

\begin{thm}\label{T:U sublinear} For all $n_1,n_2,m\in\N$, we have
$U(n_1+n_2,m) \le U(n_1,m) + U(n_2,m)$.
\end{thm}

\begin{proof}
Since $U(\cdot,m)$ is unchanged as $m$ varies over a dyadic block, it
suffices to assume that $m=2^k$ for some $k\ge 0$. We will show that
$U(n,2^k)=\mu'(n,2^k)$, where $\mu'(n,2^k)=\mu_{A'(n,2^k)}$ and $A'(n,2^k)$
is an appropriate set of wiring matrices $W\in M(n,n;\F_2)$ that has the
following closure property: if $W_i\in A'(n_i,2^k)$ for $i=1,2$, then the
block diagonal matrix
$$
W = \begin{pmatrix} W_1 & 0 \\ 0 & W_2 \end{pmatrix}\,\qquad
$$
lies in $A'(n_1+n_2,2^k)$. In terms of wirings, this just says that a
disjoint union of two wirings in this class for given $m=2^k$ also lies in
this class for the same value of $m$. Sublinearity then follows easily from
the definition of $\mu'$ and this closure property.

To define $A'(n,2^k)$, we first define $W_j$ for each $j\ge 0$ to be some
wiring in $A^*(2^{j+1}-1,2^j)$ such that $M(W,0)=\mu^*(2^{j+1}-1,2^j)=2^j$;
this exists by \rf{T:m=2^k}. We now define $A'(n,2^k)$ to be the collection
of matrices $W\in M(n,n;\F_2)$ that are of block diagonal form with $n_j\ge
0$ diagonal blocks of type $W_j$ for some $0\le j\le k$. Thus,
$n=\sum_{j=0}^k n_j(2^{j+1}-1)$.

We first show that $\mu'(n,2^k)\le U(n,2^k)$. Taking $k=0$, it is clear that
$U(n,1)=\mu'(n,1)=n$; note that the identity matrix is the only element of
$A'(n,1)$. We therefore assume that $k>0$.

If $n<2^{k+1}-1$, then $U(n,2^k)=U(n,2^{k-1})$, so by choosing $W\in
A'(n,2^{k-1})$ satisfying $M(W,0)=U(n,2^{k-1})$, we get
$$ \mu'(n,2^k) \le \mu'(n,2^{k-1}) \le U(n,2^{k-1}) = U(n,2^k)\,. $$
If $n=2^{k+1}-1$, then $\mu'(n,2^k)\le M(W_k,0)=2^k=U(n,2^k)$.

Finally if $n>2^{k+1}-1$, then we simply write $n=q(2^{k+1}-1)+r$ in the
usual way and select a wiring $W$ that decomposes into $q$ copies of $W_k$
plus one copy of a wiring $W'\in A'(r,2^k)$ satisfying $M(W',0)=U(r,2^k)$ to
deduce that $\mu'(n,2^k)\le q2^k+U(r,2^k)=U(n,2^k)$, as required.

We now prove the opposite inequality by induction. As mentioned above, the
case $k=0$ is clear. Suppose that $\mu'(\cdot,2^j)=U(\cdot,2^j)$ for all
$0\le j<k$, and suppose that $\mu'(n',2^k)=U(n',2^k)$ for all $n'<n$,
$n'\in\N$. Let $W\in A'(n,2^k)$ be such that $M(W,0)<U(n,2^k)$. Since
$U(n,2^k)\le U(n,2^{k-1})=\mu'(n,2^{k-1})$, $W$ must have a vertex of degree
$2^k$ which is part of a $W_k$. Let $W'$ be the subwiring obtained from $W$
by removing this $W_k$, and so $W'\in A'(n-2^k,2^k)$. Now $M(W_k,0)=2^k$ and
by minimality of $n$, we have $M(W',0)\ge U(n-2^k,2^k)$, so $M(W,0)\ge
U(n-2^k,2^k)+2^k\ge U(n,2^k)$, as required.
\end{proof}

Since $U(n,m)=a(n)$ whenever there exists $k\in\N$ such that $n/2<2^k\le m$,
it follows that $(a(n))$ is also sublinear, a fact we now record.

\begin{cor}\label{C:f sublinear} For all $n_1,n_2\in\N$, we have
$a(n_1+n_2) \le a(n_1) + a(n_2)$.
\end{cor}


\section{Results for $m$ near a power of $2$} \label{S:m near 2^k}
\subsection{}
In this section, we give some results for $2^k-2\le m\le 2^k+1$. Our first
result follows rather easily from \rf{T:m=2^k}.

\begin{prop}\label{P:2^{k+1}-1}
For all $\/2\le k\in\N$, we have $\mu(\cdot,2^k-1)=\mu(\cdot,2^k-2)$.
\end{prop}

\begin{proof}
Let $m:=2^k-1$. By \rf{P:general m}, we have $\mu(n,m-1)\le 2^{k-1}+1\le m$
for all $n\le 2^k$. It follows that $\mu(n,m)=\mu(n,m-1)$ for $n\le 2^k$,
since any wiring $W$ with a degree $m$ vertex satisfies $M(W,0)\ge m$.

Suppose inductively that $\mu(n',m)=\mu(n',m-1)$ for all $1\le n'<n$, where
$n>2^k$, and we wish to extend this equation to $n'=n$. By
\rf{L:alternative}, either $\mu(n,m)=\mu(n,m-1)$ and we have established the
inductive step, or
$$ \mu(n,m) \ge \mu(n-m,m)+2^{k-1} = \mu(n-m,m-1)+2^{k-1} $$
and so
\begin{align*}
\mu(n-m,m-1) + 2^{k-1} &\le \mu(n,m) \\
&\le \mu(n,m-1) \tag{trivial estimate}\\
&\le \mu(n-m,m-1)+\mu(m,m-1) \tag{sublinearity}\\
&\le \mu(n-m,m-1)+\mu(m,2^{k-1}) \tag{trivial estimate}\\
&= \mu(n-m,m-1)+2^{k-1}\,. \tag{by \rf{T:m=2^k}}
\end{align*}
The inductive step, and so the lemma, follows from equality of the first and
last lines.
\end{proof}


\begin{thm}\label{T:m=2^k+1}
Let $m=2^k$ for some $k\in\N$, and suppose that $\mu(\cdot,m-1) =
U(\cdot,m-1)$. Then $\mu(\cdot,p) = U(\cdot,p)$ also holds for $p=m$ and
$p=m+1$. In particular, $\mu(\cdot,m)=\mu(\cdot,m+1)$.
\end{thm}

\begin{proof}
\rf{P:general m} tells us that $\mu(\cdot,\cdot)\le U(\cdot,\cdot)$, so we
must prove inequalities in the opposite direction. For $k=1$, the desired
conclusion follows from Theorems \ref{T:m=2} and \ref{T:m=3}, so we assume
that $k>1$.

We first consider the case $p=m$. Suppose for the sake of contradiction that
$m=2^k>2$ is such that $\mu(\cdot,m)\ne U(\cdot,m)$, even though
$\mu(\cdot,m-1)=U(\cdot,m-1)$. Also for the sake of contradiction, assume
that $n\in\N$ is the smallest number such that $\mu(n,m)<U(n,m)$, and that
$W\in A(n,m)$ is such that $M(W,0)<U(n,m)$. Since $\mu(n,m-1)=U(n,m-1)\ge
U(n,m)$, it follows that $W$ must contain a vertex of degree $m$. This
certainly implies that $M(W,0)\ge m=U(2m-1,m)$, so $n\ge 2m$.

We write $n=q(2m-1)+r$, where $q,r\in\N$ and $r<2m$. By induction we have
$M(W,0)<qm+\mu(r,m)$. We may also assume that we cannot increase the number
of $F_m$ subgraphs in $W$ by any amount of pivoting; recall that an $F_m$ is a
forward invariant augmented complete subgraph on $m$ vertices.

We now carry out what for later reference we call a {\it Partition by Degree
argument}: we partition the set of $n$ vertices into subsets $A$ and $B$,
where $A$ consists of all vertices that lie in an $F_m$, and $B$ consists of
all other vertices. Let us write $n_A,n_B$ for the cardinalities of $A$ and
$B$, respectively.

Since we cannot increase the number of $F_m$ subgraphs by pivoting, we have $W_B\in
A(n_B,m-1)$. Since we can light all vertices in $A$ by pressing one vertex
in every $F_m$, we must have
$$ n_A < U(n,m) = qm + \mu(r,m) \le (q+1)m\,. $$
But $n_A$ is a multiple of $m$, so $n_A\le qm$. Alternatively, we can first
light at least $\mu(n_B,m-1)$ of the $B$-vertices followed by at least
$\nu(n_B,m)\ge n_A/2$ of the $A$-vertices, and so
\begin{equation}\label{E:BA bound1}
\mu(n_B,m-1)+n_A/2<qm+\mu(r,m)\,.
\end{equation}

Suppose $n_A=qm$, and so $n_B=n-qm=q(m-1)+r$. By assumption,
$\mu(n_B,m-1)=qm/2+\mu(r,m-1)$, and so
$$ \mu(n_B,m-1) + n_A/2 = qm + \mu(r,m-1)\ge qm + \mu(r,m)\,, $$
contradicting \rf{E:BA bound1}. If $n_A$ is smaller than $qm$, it must be
smaller by $q'm$ for some $q'\in\N$, thus increasing $\mu(n_B,m-1)$ by at
least $q'm/2$:
$$
\mu(n-qm-q'm,m-1) \ge \mu(n-qm-q'(m-1),m-1) = \mu(n-qm,m-1)+{\frac{q'm}2}\,.
$$
Thus, $\mu(n_B,m-1)+n_A/2$ is at least as large as in the case $n_A=qm$, and
we still get a contradiction.

We next prove that $\mu(n,m+1)=\mu(n,m)$. Again for the sake of contradiction,
we suppose that $m=2^k>2$ is such that $\mu(\cdot,m+1)\ne U(\cdot,m+1)$,
even though $\mu(\cdot,p)=U(\cdot,p)$ when $p=m-1$. This last equation holds
also for $p=m$ by the first part of the proof. Note that
$U(n,m+1)=U(n,m)=\mu(n,m)$.

Suppose also for the sake of contradiction that $n$ is minimal for the
inequality
$$ \mu(n,m+1)<U(n,m+1)=\mu(n,m)\,. $$
Now, $U(n,m+1)\le m+1$ for $n\le 2m$, so as in the first part of the proof, we
must have $n>2m$. We again write $n=q(2m-1)+r$, where $q,r\in\N$ and $r<2m$.
Let $W\in A(n,m+1)$ be such that $M(W,0)=\mu(n,m+1)$, and we assume that the
number of $F_{m+1}$s cannot be increased by pivoting, and that the only
possible pivoting operations that may increase the number of $\hat K_m$ subgraphs are
those that decrease the number of $F_{m+1}$ subgraphs; recall that a $\hat K_m$ is an
augmented complete subgraph on $m$ vertices (which is not necessarily forward
invariant).

We carry out another Partition by Degree argument, with $A$ consisting of
all vertices that lie in a $\hat K_m$ or an $F_{m+1}$, and $W_B\in A(n_B,m-1)$.
Now, $W$ must be a vertex of degree $m+1$, since $M(W,0)<\mu(n,m)$, and so
$W$ contains at least one $F_{m+1}$. Suppose that there are at least two
$F_{m+1}$s. We can light at least $\mu(n-2m-2,m+1)$ of the other vertices,
followed by at least $2\nu(m+1)=m+2$ of the vertices in the pair of
$F_{m+1}$s. Now
\begin{align*}
\mu(n-2m-2,m+1)+m+2 &= U(n-2m-2,m+1)+m+2 \tag{minimality of $n$}
  \\
&= U(n-3,m)+2 \\
&= \mu(n-3,m)+\mu(3,m)\\
&\ge \mu(n,m) \tag{sublinearity}\,,
\end{align*}
contradicting the fact that $\mu(n,m+1)<\mu(n,m)$.

Thus, there is precisely one $F_{m+1}$, and $n_A$ is equivalent to $1$ mod
$m$. We distinguish between those $\hat K_m$s that are forward invariant, which
we denote as usual by $F_m$, and those that are not, which we denote by
$N_m$. The one external link of each $N_m$ is to the $F_{m+1}$, since
otherwise we could pivot to get a second $F_{m+1}$. Furthermore, any two
$N_m$s must link to the same vertex in the $F_{m+1}$, since if this were not
the case, we could pivot about a vertex in one $N_m$ to get a wiring with one
$N_m$ linked to a second $N_m$, which in turn links to a $F_{m+1}$, and such
a configuration would allow us to get a second $F_{m+1}$ by pivoting about
the vertex in the first $N_m$.

It follows that we can light all except possibly one of the vertices in $A$,
and so $n_A-1 < qm+\mu(r,m-1)\le (q+1)m$, which self-improves to $n_A\le
qm+1$. Alternatively, as in the first part of the proof, we get
\begin{equation}\label{E:BA bound2}
\mu(n_B,m-1)+(n_A+1)/2<qm+\mu(r,m-1)\,.
\end{equation}
Suppose $n_A=qm+1$, and so $n_B=q(m-1)+r-1$. By the inductive hypothesis,
$\mu(n_B,m-1)=qm/2+\mu(r-1,m-1)$, and so by Lemma \ref{L:F},
$$ \mu(n_B,m-1) + (n_A+1)/2 = qm + \mu(r-1,m-1)+1 \ge qm + \mu(r,m)
=qm+\mu(r,m-1), $$
contradicting \rf{E:BA bound2}. The case where $n_A$ is smaller than $qm$ is
ruled out as in the first part of the proof.
\end{proof}

\subsection{Partition by degree arguments}
Since we will be seeing other variations of the above {\it Partition by
Degree arguments}, let us describe the common features of these arguments, so
that we can be sketchy in all subsequent uses of it. Given a wiring $W$ on
$n$ vertices, we partition the set of vertices into two subsets, typically
called $A$ and $B$, and we denote the cardinality of $A$ and $B$ by $n_A$
and $n_B$, respectively. The wiring will be initially pivoted so that $A$ is
forward invariant and $W_A$ will consist only of $\hat K_j$s for various $j\ge
2^k$. There will be very few links between different $\hat K_j$s in $A$, allowing
us to light almost all except at most $n_0$ of the vertices in $A$ by
pressing one vertex in each $\hat K_j$; for instance, $n_0$ was either $0$ or $1$
in the two Partition by Degree arguments in the above proof. This gives the
bound $n_A\le K-n_0$, where $K$ equals either $\mu(n,m)$ or an assumed value
of $\mu(n,m)$ from which we wish to derive a contradiction. By the structure
of $A$, we often know that $n_A$ has a certain value mod $2^k$, allowing us
to improve the estimate $n_A\le K-n_0$ to $n_A\le n_1$ for some $n_1\le
K-n_0$.

By somehow maximizing the number of $\hat K_j$s in $A$, we arrange for the
restricted wiring $W_B$ to lie in $A(n_B,m')$ for some $m'\le 2^k-1$, so we
may light at least $\mu(n_B,m')$ of these vertices followed by at least
$\nu(n_A)$ vertices in $A$. This gives the inequality
\begin{equation}\label{E:n_B n_A}
\mu(n_B,m') + \nu(n_A)\le K\,
\end{equation}

The aim of the Partition by Degree argument is now either to derive a
contradiction, or to show that $n_A=n_1$. To do this, we first consider the
possibility that $n_A=n_1$, and we typically deduce that
$\mu(n-n_1,m')+\nu(n_1)$ either equals or exceeds $K$. If instead we allow
$n_A$ to decrease below $n_1$, then $n_A$ typically must be decreased by a
multiple of $2^k$, and $\mu(n_B,m')$ increases by at least as much as
$\mu(n_A)$ decreases. Thus, if $\mu(n-n_1,m')+\nu(K_1)>K$, we get a
contradiction to \rf{E:n_B n_A} also for any value of $n_A$ less than $n_1$,
and we are done. In other instances of this argument,
$\mu(n-n_1,m')+\nu(K_1)=K$, but taking a value of $n_A$ smaller than $K_1$
increases $\mu(n_B,m')$ strictly more than $\mu(n_A)$ decreases, so we
conclude that $n_A$ must equal $n_1$ and $n_B=n-n_1$, as we are seeking to
prove in such instances. \bigskip

\subsection{A technical lemma}
We now give a lemma which makes no mentions of wirings and
vertices but which we will need later. In this lemma, $|u|$ denotes the
Hamming norm of a vector $u\in \F_2^N$, as defined in \rf{S:introduction}.

\begin{lem}\label{L:switches}
Let $n$, $N$ and $M$ be positive integers.
Then the following are equivalent:
\\
(1)
There exist vectors $a_j=(a_{i,j})_{i=1}^N\in \F_2^N$, $1\le j\le n$ such
that
\begin{equation}\label{E:switches}
\left|\sum_{j=1}^n \la_j a_j\right|=M\in\N\,,
  \qquad\text{ for all }\la=(\la_j)\in F_n:=\F_2^n\setminus\{0\}\,.
\end{equation}
\\
(2)
$M=2^{n-1}q$ for some $q\in\N$, and $N\ge 2M-2^{1-n}M$.
\\
Assuming these conditions are fulfilled, all solutions $(a_{i,j})$ to
\rf{E:switches} are equivalent modulo permutations of the $i$ and $j$
indices.
\end{lem}

\begin{proof}
Assuming the conditions (2) are fulfilled, with $M=2^{n-1}q$, we see that $N\ge
(2^n-1)q$, so we can allocate $(2^n-1)q$ vertices into $2^n-1$ pairwise
disjoint sets of $q$ vertices each. We label these sets $S_k$ for $1\le k\le
2^n-1$, and write $S=\bigcup_{i=1}^{2^n-1} S_k$. Writing $d_{n-1;k}\dots
d_{1;k}d_{0;k}$ for the binary expansion of $1\le k\le 2^n-1$, we let
$a_{i,j}:=d_{j-1;k}$ for all $i\in S_k$, and $a_{i,j}=0$ of $i\notin S$. It
is readily verified that \rf{E:switches} holds with this choice of
$(a_{i,j})$.

Conversely, suppose that $A:=(a_{i,j})$ satisfy \rf{E:switches}. Note that
this condition implies the same condition with $n$ replaced by any number
$1\le n'\le n$, and if we take $n'=n-1$, we can replace $a_{n-1}$ by either
$a_n$ or $a_{n-1}+a_n$ and the condition remains true. For each $u=(u_j)\in
F_n$, we write $S_n(u;A)$ for the set of indices $1\le i\le n$ such that
$a_{i,j}=u_j$ for all $j$. Trivially, such sets $S_n(u;A)$ are pairwise disjoint. Writing
$\#(\cdot)$ for set cardinality, we claim that $\#(S_n(u;A))=2^{1-n}M$.
Since $\#(F_n)=2^n-1$, it follows from the claim that $N\ge 2M-2^{1-n}M$.
Also, the fact that $\#(S_n(u;A))$ is independent of $u\in F_n$ means that
there is essentially only one such solution, modulo permutations of the
indices, so the result follows from the claim. We prove this by induction on
$n$.

If $n=1$, the claim is trivial. For $n=2$, note that $|a_1+a_2| =
|a_1|+|a_2|-2K = 2M-2K$, where $K$ is the number of indices $i$ for which
$a_{i,1}=a_{i,2}=1$. Since $2M-2K=M$, we must have $K=M/2$. This readily
implies the result for $n=2$.

Suppose inductively that the result is true for $n<m$, where $m>2$, and we
want to prove it for $n=m$. Let us define the following matrices
$$
A_1 = (a_{i,j})_{\substack{1\le i\le N\\1\le j\le m-2}}\,, \quad
A_2 = (a_{i,j})_{\substack{1\le i\le N\\1\le j\le m-1}}\,, \quad
A_3 = (b_{i,j})_{\substack{1\le i\le N\\1\le j\le m-1}}\,, \quad
A_4 = (c_{i,j})_{\substack{1\le i\le N\\1\le j\le m-1}}\,,
$$
where
\begin{align*}
b_{i,j} &=
  \begin{cases}
  a_{i,j}, &j\le m-2\,, \\ a_{i,m}, &j=m-1\,,
  \end{cases} \\[4pt]
c_{i,j} &=
  \begin{cases}
  a_{i,j}, &j\le m-2\,, \\ a_{i,m-1}+a_{i,m}, &j=m-1\,,
  \end{cases}
\end{align*}
We assume that $A:=(a_{i,j})_{\substack{1\le i\le N\\1\le j\le m}}$
satisfies \rf{E:switches} for $n=m$, so certainly $A_s$ satisfies
\rf{E:switches} for $1\le s\le 4$ (for $n=m-2$ or $n=m-1$).

By the inductive assumption $\#(S_{m-1}(u;A_s))=2^{2-m}M$ for each $u\in
F_{m-1}$ and $2\le s\le 4$. Considering separately those
$u=(u',u_{m-1},u_m)\in F_{m-2}\times F_1\times F_1=F_m$ such that $u'\ne 0$
and $u'=0$, we can in both cases argue as for $n=2$ above that
$\#(S_m(u))=2^{1-n}M$, as required.
\end{proof}

\subsection{Towers again}
We return to the special wirings $V_k$ discussed in Subsection \ref{SS:towers}.
For the rest of this section, a $\K_i$ will mean a $\hat K_{2^i}$, 
i.e.~an augmented complete
graph on $2^i$ vertices.  
The wiring $V_k\in A^*(2^{k+1}-1,2^k)$ consists of augmented 
complete
subgraphs $\K_0$, $\K_1$, $\K_2$,$\ldots$,$\K_k$, such that each vertex of each
$\K_i$ toggles zero vertices in $\K_j$ for $j<i$ and toggles $2^{j-1}$
vertices in each $\K_j$, $i<j\le k$. In addition the set of
vertices toggled in $\K_j$ by any $\K_i$ vertex for $i<j$ is independent of
which vertex in $\K_i$ is chosen, so we may as well restrict ourselves to
considering vertex press sets where we are allowed to press only one vertex,
which we call the {\it designated vertex}, in each $\K_i$. We say that $\K_i$
is {\it activated} if its designated vertex is pressed. Thus, each $\K_i$ can
be viewed as a single switch which toggles $2^{j-1}$ indices in $\K_j$ for
each $j>i$. We assume that this wiring is arranged so that activating one or
more of the $\K_i$, $i<j$, always lights exactly $2^{j-1}$ of the vertices in
$\K_j$. This is possible by \rf{L:switches}. In view of the uniqueness in
\rf{L:switches}, this defines the wiring $V_k$ uniquely up to relabeling of
the vertices within each $\K_j$, $1\le j\le k$.

The next theorem generalises 
our remark that the Sylvester-Hadamard wiring $W_k$ can be
pivoted to obtain $V_k$. In
this theorem, we push further with the ideas in the proof of \rf{T:m=2^k+1}
to show that if $\mu(\cdot,p) = U(\cdot,p)$ for $p<2^k$ then, modulo (full
and partial) pivoting, there really is only one optimal wiring in $A(n,2^k)$
for each $n=q(2^{k+1}-1)$, $q\in\N$, namely $q$ disjoint copies of $V_k$.

\begin{thm}\label{T:n=q(2m-1)}
Suppose that $\mu(\cdot,p) = U(\cdot,p)$ for all $p<m:=2^k$, for some
$k\in\N$. If 
$n=q(2m-1)$
for some $q\in\N$, 
and if $W\in A(n,m)$ is such that $M(W,0)=\mu(n,m)$, 
then $W$ can be pivoted to the block diagonal wiring
$\diag(V_k,\dots,V_k)$, where $V_k\in A^*(2m-1,m)$ is as above; both full
and partial pivoting operations may be required.
\end{thm}

\begin{proof}
We will construct a chain of restricted wirings, so let us write $W_k$ in
place of $W$ for our initial wiring, and we also write $N_k$ in place of
$n$. We assume without loss of generality that $W_k$ has the property that
no additional $\K_k$ subgraphs can be obtained by pivoting. We denote by $B_k$ the set
of all $N_k$ vertices.

We do a Partition by Degree argument, partitioning the $N_k$ vertices into
two sets: $A_k$, of cardinality $n_k$, contains all vertices in any $\K_k$,
and $B_{k-1}$, of cardinality $N_{k-1}$, contains all the other vertices. We
denote by $W_{k-1}$ the wiring $W_k$ restricted to $B_{k-1}$. Then,
$W_{k-1}\in A(N_{k-1},2^k-1)$, since otherwise we could create an extra
$\K_k$ by pivoting.

As usual, we have $n_k\le\mu(n,2^k)=q2^k$ and 
\begin{equation}\label{E:n=q(2m-1)}
\mu(N_{k-1},2^k-1)+{\frac{n_k}2} \le \mu(n,2^k) = q2^k\,. 
\end{equation}
The first inequality forces $n_k\le q2^k$, so $N_{k-1}\ge q(2^k-1)$. If
$N_{k-1}=q(2^k-1)$, we get equality in \rf{E:n=q(2m-1)}, but this inequality
cannot hold if $n_k<q2^k$, since it would force the inequality
$U(i2^k,2^{k-1})\le i2^{k-1}$ for some $i\in\N$, which itself can be reduced
to $U(i,2^{k-1})\le 0$, $i\in\N$, which we know to be false. Thus, the only
possible value for $(n_k,N_{k-1})$ is $(q2^k,q(2^k-1))$.

The fact that this choice of $(n_k,N_{k-1})$ only satisfies \rf{E:n=q(2m-1)}
with equality means that we can analyze the wiring more closely and rule out
any wiring that creates any slippage in the left-hand side bounds. In
particular, if there were a vertex in $W_{k-1}$ of degree $j>2^{k-1}$, we
could pivot about it relative to $A_k$ to get a $\hat K_j$. 
	Vertices in the $\hat K_j$
have at most $2^k-j$ links outside the $\hat K_j$, which must all be in $A_k$
because of the pivoting process).  By pressing a $\hat K_j$ vertex and then one
vertex in every $\K_k$, we light all the vertices in the $\hat K_j$ and all except
at most $2^k-j$ vertices in $A_k$, thus giving a contradiction since
$q2^k-(2^k-j)+j>q2^k$. A $\K_{k-1}$ also leads to a contradiction unless its
vertices link to exactly $2^{k-1}$ vertices in $A_k$.

Thus, $W_{k-1}\in A(q(2^k-1),2^{k-1})$ and \rf{E:n=q(2m-1)} forces
$M(W_{k-1},0) = \mu(q(2^k-1),2^{k-1})$. Thus, $W_{k-1}$ satisfies assumptions
similar to those of $W_k$, but with $k$ replaced by $k-1$. We can continue
this process, creating a chain of restricted wirings $W_j$ and associated
partition sets $A_j$ consisting of the vertices in $q$ copies of $\K_j$ and
$B_j$ of cardinality $N_{j-1}=q(2^j-1)$ such that $W_{j-1}:=W_{B_{j_1}}\in
A(N_{j-1},2^{j-1})$, for $j=0,\dots,k$.

Since all $\K_j$s are obtained by (partial or full) pivoting, all vertices in
any one $\K_j$ link to the same set of vertices. As in the discussion of
$V_k$ before this theorem, we may as well restrict to vertex press sets
where we are only allowed to press a single {\it designated vertex} in each
$\K_j$, and we say that $\K_j$ is {\it activated} if its designated vertex is
pressed. We also talk about a $\K_j$ being {\it switched} if its designated
vertex is one of the vertices given by a perturbation $y$ of an existing
vertex press set $x$, thus yielding a vertex press set $x+y$.

For $1\le j\le k$, we know that the designated vertex in any one $\K_{j-1}$
links to $2^{j-1}$ vertices in $A_j$. By activating every $\K_{j-1}$, and
then activating any $\K_j$s in which fewer than $2^{j-1}$ vertices are lit, we
could light strictly more than $q2^j=\mu(N_j,2^j)$ vertices in $A_j$ if
there were at least one $\K_j$ that had either strictly more, or strictly less,
than $2^{j-1}$ lit vertices after every $\K_{j-1}$ had been activated. It
follows that the links from any two different $K_{j-1}$s must be to distinct
sets of vertices in $A_j$, and that these links must be evenly distributed,
in the sense that there must be $2^{j-1}$ of them in each $\K_j$.

Suppose now that $1<j\le k$. By activating every $\K_{j-2}$ and every $\K_j$,
we light $q2^{j-2}$ vertices in $A_{j-2}$ and the same number in $A_{j-1}$,
and arguing as above we see that there must be exactly $2^{j-1}$ vertices
lit in each $\K_j$. We can continue this argument to deduce inductively that
any one vertex in $A_{j'}$ is linked to exactly $2^{j-1}$ vertices in $A_j$
if $j\ge j'$, and to no vertices in $A_j$ if $j<j'$. Furthermore there are
links to $2^{j-1}$ vertices in any given $\K_j$ from designated $\K_{j'}$
vertices whenever $j'<j$.

We next prove that all of these $\K_j$s are arranged in $V_k$s. This is
trivial if $k=1$, since each $\K_0$ has only one link to $A_1$, and each
$\K_1$ has a link from one of the $\K_0$s. Suppose inductively that all the
$A_j$s for $j\le k-1$ are arranged into $q$ copies of $V_{k-1}$. We wish to
prove the same with $k-1$ replaced by $k$.

We fix one particular $\K_{k-1}$, which we call $L_{k-1}$ and, for each $0\le
j< k-1$, denote by $L_j$ the copy of $\K_j$ that is linked to $L_{k-1}$. The
sets $L_j$, $0\le j\le k-1$ lie in a particular copy of $V_{k-1}$ that we
will call $U_{k-1}$. For each $0\le j< k-1$, let $x_j$ be the vertex press
sets where we activate every $\K_{k-1}$ other than $L_{k-1}$, and we also
activate $L_j$. By the properties of the $V_{k-1}$, this results in having
$2^{k-1}$ vertices lit in each $V_{k-1}$, and some vertices in $A_k$ are
also lit as a result of the $2^{k-1}$ links from each $\K_{k-1}$ and from
$L_j$ into $A_k$. By then activating any $\K_k$ where fewer than half of the
vertices are lit, we get at least $q(2^{k-1}+2^{k-1})=q2^k$ vertices lit in
$B_k$, the maximum amount allowed.

But we would get strictly more than this if the vertex press set $x_j$
resulted in any number of lit vertices other than $2^{k-1}$ in any $A_k$.
Thus, $x_j$ must result in $q2^{k-1}$ lit vertices in $A_k$, with exactly
$2^{k-1}$ of these in each $\K_k$. But there are only $2^{k-1}$ links from
each $\K_k$ or from $L_j$ to $A_k$, so it must be that no two of these links
are to the same vertex in $A_k$, since otherwise there would be fewer than
$q2^{k-1}$ vertices lit in $A_k$ as a result of $x_j$.

Consider more generally a vertex press set $x$ where we press the designated
vertex in $\K_{k-1}$ in all cases except $L_{k-1}$, and we also press the
designated vertex in one or more of the sets $L_j$, $0\le j\le k-1$. For all
such $x$, we get $2^{k-1}$ lit vertices in each $V_{k-1}$, so again we must
have $q2^{k-1}$ lit vertices in $A_k$, with exactly $2^{k-1}$ of these in
each $\K_k$. Since we have seen that the links to $A_k$ from $U_{k-1}$ are
disjoint from the links to $A_k$ from every $\K_{k-1}$ other than $L_{k-1}$,
it follows that any nontrivial combination of activations of the sets $L_j$,
$0\le j\le k-1$ toggles the same number of vertices in each $\K_k$ and
$2^{k-1}$ such vertices across the union all all $\K_k$s.

Denoting by $L_k$ some particular $\K_k$ where nontrivial combination of
activations of the sets $L_j$ toggle at least one vertex, we assume the
number of such toggles is $M$. Viewing our designated vertices in $L_j$ as
switches for $0\le j\le k-1$, we now apply \rf{L:switches} and the fact that
$1\le M\le 2^{k-1}$ to deduce that $M=2^{k-1}$. This uses up all the
available links from $U_{k-1}$ to $A_k$. Now $U_{k-1}$ is a fixed but
arbitrary $V_{k-1}$, so it follows that each $V_{k-1}$ is linked only to a
single $\K_k$, and so our full wiring consists of $q$ copies of $V_k$, as
required.
\end{proof}

\section{The cases $m=4,5$} \label{S:m=4}
\subsection{}
In this section, we prove \rf{T:m=4}. Throughout, an {$n$-optimal wiring} is
a wiring $W\in A(n,m)$ for which $M(W,0)=\mu(n,m)$; the parameter $m$ is in
all such cases understood.

\begin{proof}[Proof of \rf{T:m=4}]
Part (a) can be restated as $\mu(\cdot,p)=U(\cdot,p)$ for $p=4,5$. It is
readily verified from \rf{T:m=3}(a) that $\mu(\cdot,p)=U(\cdot,p)$ when
$p=3$, so it extends to $p=4,5$ by \rf{T:m=2^k+1}.

We now prove Part (b). The desired formula for $\mu^*(n,4)-4k$, $n=7k+i\ge
4$, is given by $a_i$ in the following table:

{\renewcommand{\arraystretch}{1.2}\renewcommand{\tabcolsep}{5mm}
\begin{center}
\begin{tabular}{c|c|c|c|c|c|c|c}
$i$   & 1 & 2 & 3 & 4 & 5 & 6 & 7 \\ \hline
$a_i$ & 2 & 2 & 2 & 4 & 4 & 4 & 4 
\end{tabular}
\end{center}
}

It is readily verified that $a_i$ equals the least even integer not less
than $\mu(7k+i,4)-4k$. Since pressing any vertex for a wiring in $A^*(n,4)$
preserves the parity of the number of lit vertices, $\mu(7k+i,4)$ must be
even. Thus $\mu^*(n,4)\ge 4k+a_i$.

We now prove the converse by induction. The nontrivial part is to prove it
for $4\le n\le 10$. Once this is proved, it follows inductively for all
$n=7k+i>10$ using \rf{E:sublinear}:
$$
\mu^*(7k+i,4)\le \mu^*(7(k-1)+i,4)+\mu^*(7,4) \le (4(k-1)+a_i)+4=4k+a_i\,.
$$

It remains to prove that $\mu^*(n,4)\le 4k+a_i$ when $4\le n\le 10$.
Trivially $\mu^*(4,4)=4$  and
\begin{alignat*}{3}
\mu^*(5,4) &\le \mu^*(4,3)+1 = 4   &&\qquad\text{(\rf{L:nn'm})}\,, \\
\mu^*(6,4) &\le 2\mu^*(3,2)  = 4   &&\qquad\text{(\rf{L:reflect})}\,, \\
\mu^*(7,4) &= 4                    &&\qquad\text{(\rf{T:m=2^k})}\,, \\
\mu^*(8,4) &\le \mu^*(7,3) + 1 = 6 &&\qquad\text{(\rf{L:nn'm})}\,, \\
\mu^*(9,4) &\le \mu^*(8,3) + 1 = 7 &&\qquad\text{(\rf{L:nn'm})}\,.
\end{alignat*}
All except the last of these is sharp, and parity considerations allow us to
improve the last one to the sharp $\mu^*(9,4)\le 6$.

Finally, $\mu^*(10,4)\le 6$ follows by consideration of the wiring
$$
W_{10} =
  \begin{pmatrix}
  1&0&0&0&0&0&0&0&0&0 \\ 1&1&1&0&0&0&0&0&0&0 \\ 0&1&1&0&0&0&0&0&0&0 \\
  1&0&0&1&1&1&0&0&0&0 \\ 0&1&1&1&1&1&0&0&0&0 \\ 0&0&0&1&1&1&0&0&0&0 \\
  1&0&0&0&0&0&1&1&1&1 \\ 0&1&1&0&0&0&1&1&1&1 \\ 0&0&0&1&1&1&1&1&1&1 \\
  0&0&0&0&0&0&1&1&1&1 \\
  \end{pmatrix}
$$
	(cf. Figure \ref{F:2-7}
\begin{figure}[h]
	\begin{center}
		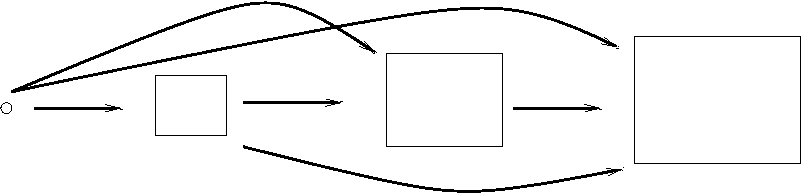
		\caption{$W_{10}$}\label{F:2-7}
	\end{center}
\end{figure}
All columns except columns 1, 2, 4, and 7 are duplicates of these columns,
so we can restrict ourselves to sets of vertex presses involving only these
four vertices. With this restriction, we can proceed to list all sixteen
possible values of $x$, and deduce that $M(W_{10},0)=6$.
\end{proof}

\subsection{} Let us mention an alternative, more instructive, way of proving that
$M(W_{10},0)\le 6$. Again, we may restrict ourselves to pressing only some
combination of vertices 1,2,4, and 7. Note first that $W_{10}$ consists of
one copy each of a $\hat K_1$, $\hat K_2$, $\hat K_3$, and $\hat K_4$ (vertices 1, 2--3, 4--6,
and 7--10, respectively), and $\hat K_i$ is connected to $\hat K_j$ only if $i<j$. The
subwiring for vertices 1--3 is such that we can never light all three
vertices (by parity, since all vertices have degree 2), and to get only one
unlit vertex, we must press vertex 1 and/or vertex 2. But all three of these
possibilities throws both the $\hat K_3$ and $\hat K_4$ out of sync since the links
from vertices 1 and 2 into the $\hat K_3$ are different from each other, and
similarly for the links into the $\hat K_4$. Furthermore, the $\hat K_3$ and $\hat K_4$
vertices remain out of sync regardless of whether we press vertices $4$,
$7$, or both. Thus, the unlit vertices always include either all of 1--3, or
at least one vertex each from 1--3, 4--6, and 7--10. Thus, $M(W_{10},0)\le 7$
and parity considerations improve this to $M(W_{10},0)\le 6$.

\subsection{Questions}
So far, we know this:
\begin{thm}\label{T:m<=5} For $m\in\N$, $m\le 5$, we have $\mu(\cdot,m) =
U(\cdot,m)$.
\end{thm}

This naturally prompts the following question, which we cannot
answer.

\begin{que}\label{Q:general m}
Is it true that $\mu(\cdot,m)=U(\cdot,m)$ when $m>5$?
\end{que}

\rf{T:m=2^k} states that if $m$ is a power of $2$, then there exists wirings
for $(n,m)=(2m-1,m)$ that are optimal in the sense that $\mu^*(n,m)$ and
$\mu(n,m)$ both equal to $(n+1)/2$ (the smallest possible value for
$\mu(n,m)$ according to \rf{L:mean}). This is evidence that powers of $2$
are significant boundaries for the behavior of $m\mapsto\mu(\cdot,m)$ and
$m\mapsto\mu^*(\cdot,m)$. This fact motivates the following pair of open
questions with which we close the article. Note that the first one is simply a
weaker version of \rf{Q:general m}.

\begin{que}
Is it true that $\mu(n,m)$ is independent of $m$ for all $2^k\le m\le
2^{k+1}-1$, $k\in\N$?
\end{que}

The answer to the above question is affirmative if we restrict to $m\le 5$.

\begin{que}
Is it true that $\mu(n,m_1)-\mu^*(n,m_2)$ is bounded independent of
$n,k\in\N$ for all $2^k\le m_1,m_2\le 2^{k+1}-1$, $k\in\N$?
\end{que}

The answer to this last question is affirmative if we restrict to $m_1,m_2\le 4$.


\end{document}